\documentclass[12pt]{amsart}
\usepackage{amssymb,amsmath,amsthm,mathrsfs}

\newtheorem{thm}{Theorem}[section]
\newtheorem{lemma}[thm]{Lemma}
\newtheorem{cor}[thm]{Corollary}

\numberwithin{equation}{section}
\theoremstyle{definition}

\newtheorem{ex}[thm]{Example}

\newtheorem*{remark}{Remark}

\newcommand{\SL}{\mathrm{SL}}

\newcommand{\bbZ}{\mathbb{Z}}

\begin{document}
\title[An Equivalence Relation on A Set of Words of Finite Length]
{An Equivalence Relation on \\ A Set of Words of Finite Length}
\author{Yotsanan Meemark and Tassawee Thitipak}
\address{Yotsanan Meemark\\ Department of Mathematics\\ Faculty of Science
\\ Chulalongkorn University\\ Bangkok, 10330 THAILAND}
\email{\tt yotsanan.m@chula.ac.th}

\address{Tassawee Thitipak\\ Department of Mathematics\\ Faculty of Science
\\ Chulalongkorn University\\ Bangkok, 10330 THAILAND}
\email{\tt thitipakt@gmail.com}

\thanks{The research of the first author was supported in part by Grants for Development of New Faculty Staff from Chulalongkorn University, Thailand. This work grows out of the second author's master thesis at Chulalongkorn university written under the direction of the first author to which the second author expresses his gratitude.} 



\begin{abstract}
In this work, we study several equivalence relations induced from the partitions of the sets of words of finite length. We have results on words over finite fields extending the work of Bacher (2002, Europ. J. Combinatorics, {\bf 23}, 141-147). Cardinalities of its equivalence classes and explicit relationships between two words are determined. Moreover,  we deal with words of finite length over the ring $\mathbb{Z}/N\mathbb{Z}$ where $N$ is a positive integer. We have arithmetic results parallel to Bacher's.
\end{abstract}

\maketitle

\section{Introduction}

Let $k$ be a finite field and $F_k$ denote the set of all finite words with letters in $k$. $F_k$ is a free monoid with identity $\varepsilon$, called the {\it empty word}. Consider the special linear group of degree two over $k$, $\SL_2(k)$, consisting of $2 \times 2$ matrices over $k$ of determinant one. It has been proved in \cite{B02} Lemma 2.1 that $\SL_2(k)$ generated as a monoid by the set of matrices 
\[
S =  \left\{ \begin{bmatrix} 0 & 1 \\ -1 & \alpha \end{bmatrix} : \alpha \in k \right\}.
\]
We can view $S$ as $k$ and thus every word $w = \alpha_1 \dots \alpha_l \in F_k$ is corresponding to the product 
\[
\begin{bmatrix} 0 & 1 \\ -1 & \alpha_1 \end{bmatrix} \dots
\begin{bmatrix} 0 & 1 \\ -1 & \alpha_l \end{bmatrix} \in \SL_2(k). 
\]
This gives rise to an onto homomorphism of monoids
\[
\pi : F_k \to \SL_2(k).
\]

We define an equivalence relation $\sim$ on $k^2 \smallsetminus \left\{\begin{bmatrix} 0 \\ 0 \end{bmatrix} \right\}$ by 
\[
\begin{bmatrix} s \\ t \end{bmatrix} \sim \begin{bmatrix} u \\ v \end{bmatrix}
\Leftrightarrow \; \text{$\begin{bmatrix} s \\ t \end{bmatrix} = \lambda \begin{bmatrix} u \\ v \end{bmatrix}$ for some $\lambda \in k^\times$}.
\]
Its equivalence classes are the lines spanned by $\begin{bmatrix} 1 \\ x \end{bmatrix}$, $x \in k$, and the line spanned by~$\begin{bmatrix} 0 \\ 1 \end{bmatrix}$, called the {\it infinite line}, with the origin deleted. Then we usually write these classes as $\begin{bmatrix} 1 \\ x \end{bmatrix}$, ${x \in k}$, and $\begin{bmatrix} 0 \\ 1 \end{bmatrix}$. Thus the set of all equivalence classes, denoted by $\mathbb{P}^1(k)$ and called the {\it projective {$1$-space}}. The group $\SL_2(k)$ acts on  $\mathbb{P}^1(k)$ by left multiplication. Bacher defined the subset $\mathcal{A}$ of $F_k$ by
\[
\mathcal{A} = \left\{w \in F_k :  \pi(w)\begin{bmatrix} 0 \\ 1 \end{bmatrix} 
= \begin{bmatrix} 1 \\ 0 \end{bmatrix}\right\}.
\]
The sets $\mathcal{A}$ and $\mathcal{C} = F_k \smallsetminus \mathcal{A}$ divide $F_k$ into two disjoint pieces. This partition leads to an equivalence relation on $F_k$. 

For $r \in k$, we define two disjoint subsets $\mathcal{A}_r$ and $\mathcal{C}_r$ of $F_k$ by
\[ 
\mathcal{A}_r = \left\{w \in F_k :  \pi(w)\begin{bmatrix} 0 \\ 1 \end{bmatrix} 
= \begin{bmatrix} 1 \\ r \end{bmatrix}\right\}. 
\]
and $\mathcal{C}_r = F_k \smallsetminus \mathcal{A}_r$. Hence $\mathcal{A} = \mathcal{A}_0$. 
In Sections 2 and 3, we investigate arithmetic and combinatorial properties of the equivalence relation on $F_k$ induced by the partition $\mathcal{A}_r$ and $\mathcal{C}_r$.

Let $N$ be a positive integer. Another route to extend Bacher's work is to study the special linear group over $\mathbb{Z}/N\mathbb{Z}$, the ring of integers modulo $N$. We present this topic in Section 4. Write $F_N$ for the set of all finite words with letters in $\bbZ/N\bbZ$. Consider the special linear group of degree two over $\bbZ/N\bbZ$, $\SL_2(\bbZ/N\bbZ)$, consisting  of $2 \times 2$ matrices over $\bbZ/N\bbZ$ of determinant one. Let 
\[
S' =  \left\{ \begin{bmatrix} 0 & 1 \\ -1 & \alpha \end{bmatrix} : \alpha \in \bbZ/N\bbZ \right\}.
\]
We show that this set generates $\SL_2(\bbZ/N\bbZ)$ as a monoid. Our proof is different from \cite{B02} Lemma 2.1. We use the basic fact that every closed subset of a finite group is a group. This result shows that every element of $\SL_2(\bbZ/N\bbZ)$ can be written in at least one way as a finite word with letters in $S'$.  

We can also consider $S'$ as $\bbZ/N\bbZ$ and hence every word $w = \alpha_1 \dots \alpha_l \in F_N$ is corresponding to the product 
\[
\begin{bmatrix} 0 & 1 \\ -1 & \alpha_1 \end{bmatrix} \dots
\begin{bmatrix} 0 & 1 \\ -1 & \alpha_l \end{bmatrix} \in \SL_2(\bbZ/N\bbZ). 
\]
This yields an onto homomorphism of monoids
\[
\pi : F_N \to \SL_2(\bbZ/N\bbZ).
\]
For $\mathbb{Z}/N\mathbb{Z}$, we define an equivalence relation $\sim'$ on $(\bbZ/N\bbZ)^2 \smallsetminus \left\{\begin{bmatrix} 0 \\ 0 \end{bmatrix} \right\}$ by 
\[
\begin{bmatrix} s \\ t \end{bmatrix} \sim' \begin{bmatrix} u \\ v \end{bmatrix}
\Leftrightarrow \; \text{$\begin{bmatrix} s \\ t \end{bmatrix} = \lambda \begin{bmatrix} u \\ v \end{bmatrix}$ for some $\lambda \in (\bbZ/N\bbZ)^\times$}.
\]
Here $(\bbZ/N\bbZ)^\times$  denotes the unit group of the ring $\bbZ/N\bbZ$. The group $\SL_2(\bbZ/N\bbZ)$ acts on  the set of equivalence classes by left multiplication. Parallel to Bacher's, we set 
\[
\bar{\mathcal{A}} = \left\{w \in F_N :  \pi(w)\begin{bmatrix} 0 \\ 1 \end{bmatrix} 
= \begin{bmatrix} 1 \\ 0 \end{bmatrix}\right\}
\]
and $\bar{\mathcal{C}} = F_N \smallsetminus \bar{\mathcal{A}}$. We study this partition of $F_N$ in the last two sections.

The paper is organized as follows. Arithmetic and combinatorial properties implying the cardinalities of $\mathcal{A}_r$ and $\mathcal{C}_r$ are studied in Section $2$. Section 3 gives an algorithm to distinguish the partition $\mathcal{A}_r$ and $\mathcal{C}_r$. Words over $\bbZ/N\bbZ$ and the partition $\bar{\mathcal{A}}$ and $\bar{\mathcal{C}}$ are presented in Section 4. The final section is devoted to $\bar{\mathcal{A}}$ including unique factorization, predecessors, successors and periodic words, parallel to Bacher's $\mathcal{A}_0$.

\bigskip

\section{Cardinalities of $\mathcal{A}_r$ and $\mathcal{C}_r$}

This section presents the preliminary properties of words in $\mathcal{A}_r$ and results on the cardinalities of $\mathcal{A}_r$ and $\mathcal{C}_r$.

For $w \in F_k$ with $\pi(w) = \begin{bmatrix} a & b \\c & d \end{bmatrix} \in \SL_2(k)$, we note that
\begin{align*}
w \in \mathcal{A}_r 
&\Leftrightarrow \begin{bmatrix} 1 \\ r \end{bmatrix}  = \begin{bmatrix} a & b \\c & d \end{bmatrix} \begin{bmatrix} 0 \\ 1 \end{bmatrix} = \begin{bmatrix} b \\ d \end{bmatrix} \Leftrightarrow d=br \\ 
& \Leftrightarrow \pi(w) = \begin{bmatrix}a & b \\ ar-b^{-1} & br \end{bmatrix} \; \text{with $a \in k, b \in k^{\times}$}.
\end{align*}
Therefore we have shown

\begin{thm}
For $r \in k$,
\[
\mathcal{A}_r = \left\{w \in F_k : \pi(w)=\begin{bmatrix}a & b \\ar-b^{-1} & br \end{bmatrix} \text{for some} \ a \in k, b \in k^{\times}\right\}.
\]
\end{thm}

The set $\mathcal{A}_0$ has been studied by Bacher in \cite{B02}. Our results are for the case $r \ne 0$. For $l \ge 0$, we write $F_k^l$ for the set of words over $k$ of length $l$, $\mathcal{A}_r^l = F_k^l \cap \mathcal{A}_r$ and $\mathcal{C}_r^l = F_k^l \cap \mathcal{C}_r$.  Unless specify, we assume $r \in k^\times$ throughout this section. We begin with the right insertion. 

\begin{thm} \label{rightadd}
Let $w \in F_k$. Then $w \in \mathcal{A}_r^l$ if and only if $w \alpha \in \mathcal{C}_r^{l+1}$ for all $\alpha \in k$. Moreover, if $w \in \mathcal{C}_r^l$, then there exists a unique $\alpha \in k$ such that $w \alpha \in \mathcal{A}_r^{l+1}$.
\end{thm} 
\begin{proof}
Assume that $w \in \mathcal{A}_r^l$ and let $\alpha \in k$. Then $\pi(w)=\begin{bmatrix} a & b \\ar-b^{-1} & br \end{bmatrix}$ for some $a \in k$ and $b \in k^\times$. Thus 
\[
\pi(w\alpha)=\pi(w)\pi(\alpha)=\begin{bmatrix} a & b \\ar-b^{-1} & br \end{bmatrix}\begin{bmatrix} 0 & 1 \\-1 & \alpha \end{bmatrix}
=\begin{bmatrix} -b & a+\alpha b \\ -br & ar-b^{-1}+\alpha br \end{bmatrix}.
\] 
If $ar-b^{-1}+\alpha br = (a+\alpha b)r$, then $-b^{-1} = 0$, a contradiction. Thus $w\alpha \in \mathcal{C}_r^{l+1}$. Conversely, suppose that $w \in \mathcal{C}_r^l$. Then $\pi(w)=\begin{bmatrix} a & b \\c & d \end{bmatrix} \in \SL_2(k)$ and $d \ne br$. Note that for $\alpha \in k$, we have
\[
\pi(w\alpha)=\pi(w)\pi(\alpha)=\begin{bmatrix} a & b \\c & d \end{bmatrix}\begin{bmatrix} 0 & 1 \\-1 & \alpha \end{bmatrix}=\begin{bmatrix} -b & a+\alpha b \\-d & c+\alpha d \end{bmatrix}.
\]
Since $d \ne br$, we can choose a unique $\alpha$, namely $\alpha = (ar-c)(d-br)^{-1} \in k$ such that 
$\pi(w\alpha) = \begin{bmatrix} -b & (d-br)^{-1} \\-d & r(d-br)^{-1} \end{bmatrix}$ 
and hence $w\alpha \in \mathcal{A}_r^{l+1}$.
\end{proof}
 
For the left insertion, we obtain a slightly different property. 
 
\begin{thm}
Let  $w \in F_k$ with $\pi(w)=\begin{bmatrix} a & b \\ c & d \end{bmatrix} \in \SL_2(k)$.\\ 
{\rm (i)} If $w \in \mathcal{C}_r^l$, then $d=0$ if and only if $\alpha w \in \mathcal{C}_r^{l+1}$ for all $\alpha \in k$.\\
{\rm (ii)} If $w \in \mathcal{A}_r^l$, then there exists a unique $\alpha \in k$ such that $\alpha w \in \mathcal{A}_r^{l+1}$.
\end{thm} 
\begin{proof}
We first observe that for $\alpha \in k$, $$\pi(\alpha w)=\pi(\alpha)\pi(w)=\begin{bmatrix} 0 & 1 \\-1 & \alpha \end{bmatrix}\begin{bmatrix} a & b \\c & d \end{bmatrix}=\begin{bmatrix} c & d \\-a+\alpha c & -b+\alpha d \end{bmatrix}.$$ 

\noindent {\rm (i)} Assume that $w \in \mathcal{C}_r^l$. If $d=0$, then $b \ne 0$, so $\pi(\alpha w) = \begin{bmatrix} c & 0 \\-a+\alpha c & -b \end{bmatrix}$ which means $\alpha w \in \mathcal{C}_r^{l+1}$. If $d \ne 0$, then there exists $\alpha = (b+dr)d^{-1}$ such that $\pi(\alpha w) = \begin{bmatrix} c & d \\-d^{-1}+cr & dr \end{bmatrix}$ which implies $\alpha w \in \mathcal{A}_r^{l+1}$.

\noindent {\rm (ii)} Assume that $w \in \mathcal{A}_r^l$. Then $d=br$. A simple calculation yields a unique $\alpha = r+r^{-1}$ such that $\pi(\alpha w) = \begin{bmatrix} c & br \\-a+c(r+r^{-1}) & br^2 \end{bmatrix}$ which means $\alpha w \in \mathcal{A}_r^{l+1}$.
\end{proof}

Next we present results on left and right deletions of a word $w \in \mathcal{A}_r$.
 
\begin{thm} \label{cancel}
Let $\alpha_1 \dots \alpha_l \in \mathcal{A}_r^l$. Then $\alpha_1 \dots \alpha_{l-1} \in \mathcal{C}_r^{l-1}$, and
$\alpha_2 \dots \alpha_l \in \mathcal{A}_r^{l-1}$ if and only if $\alpha_1 = r+r^{-1}$.
\end{thm} 
\begin{proof}
Assume that $\alpha_1\ldots\alpha_l \in \mathcal{A}_r^l$. Then $\pi(\alpha_1\ldots\alpha_l)=\begin{bmatrix} a & b \\ ar-b^{-1} & br \end{bmatrix}$ for some $a \in k$ and $b \in k^\times$. Thus
\begin{align*}
\pi(\alpha_1\ldots\alpha_{l-1})=\pi(\alpha_1\ldots\alpha_l) \pi(\alpha_l)^{-1}
&=\begin{bmatrix} a & b \\ ar-b^{-1} & br \end{bmatrix} \begin{bmatrix} \alpha_l & -1 \\ 1 & 0 \end{bmatrix}\\
&=\begin{bmatrix} \alpha_l a+b & -a \\ \alpha_l (ar-b^{-1})+br & b^{-1}-ar \end{bmatrix}.
\end{align*}
Since $b^{-1} \ne 0$, $b^{-1} - ar \ne -ar$ and so $\alpha_1\ldots\alpha_{l-1} \in \mathcal{C}_r^{l-1}$. Hence
\begin{align*}
\pi(\alpha_2\ldots\alpha_l)=\pi(\alpha_1)^{-1} \pi(\alpha_1\ldots\alpha_l) 
&=\begin{bmatrix} \alpha_1 & -1 \\ 1 & 0 \end{bmatrix} \begin{bmatrix} a & b \\ ar-b^{-1} & br \end{bmatrix}\\
&=\begin{bmatrix} \alpha_1 a - ar + b^{-1} & \alpha_1 b - br \\ a & b \end{bmatrix}.
\end{align*}
Therefore $\alpha_2\ldots\alpha_l \in \mathcal{A}_r^{l-1} \Leftrightarrow b = (\alpha_1 b - br)r \Leftrightarrow \alpha_1 = r + r^{-1}$.
\end{proof}

Theorem \ref{rightadd} results in $|\mathcal{A}_r^{l+1}| \ge |\mathcal{C}_r^l|$ and Theorem \ref{cancel} (i) gives rise to $|\mathcal{A}_r^{l+1}| \le |\mathcal{C}_r^l|$. Thus $|\mathcal{A}_r^{l+1}| = |\mathcal{C}_r^l|$. Since $|\mathcal{A}_r^l| + |\mathcal{C}_r^l| = q^l$, we get the recurrence relation
\[
|\mathcal{A}_r^{l+1}| + |\mathcal{A}_r^{l}| = q^l \;\, \text{for $l \ge 0$} \quad \text{and} \quad |\mathcal{A}_r^{0}| = 0.
\]
Solving this relation, we obtain the cardinalities of $\mathcal{A}_r^l$ and $\mathcal{C}_r^l$ for all $l \ge 0$. It should be pointing out that Bacher had the same numbers for $r=0$ in \cite{B02} Corollary 2.3. We record this result in

\begin{cor}
For a finite field $k$ with $q$ elements, $l \ge 0$ and $r \in k$, we have
\[
|\mathcal{A}_r^l| = \frac{q^l - (-1)^l}{q+1} \quad \text{and} \quad |\mathcal{C}_r^l| = \frac{q^{l+1} + (-1)^l}{q+1}.
\]
\end{cor}

\bigskip

\section{Induced Equivalence Relations}

Let $r \in k$. The partition $\mathcal{A}_r$ and $\mathcal{C}_r$ of $F_k$ induces the equivalence relation $\sim_r$ on $F_k$. Its properties are studied in our next theorem.

\begin{thm} \label{main}
Let $x \in F_k$ and $\beta \in k$. We have \\
{\rm (i)} If $\alpha \in k$ and $\alpha \ne r$, then $\alpha\beta x \sim_r \gamma x$ where 
$\gamma = \dfrac{r^2 - (\alpha - \beta)r + 1 - \alpha \beta}{r-\alpha}$. \\
{\rm (ii)} $r\beta x \in \mathcal{A}_r$ if and only if $x \in \mathcal{A}_0$.
\end{thm} 
\begin{proof} 
Let $\pi(x)=\begin{bmatrix}a & b \\c & d \end{bmatrix} \in \SL_2(k)$. \\
(i) Assume that $\alpha, \beta \in k$ and $\alpha \ne r$. Then
\[
\pi(\alpha\beta x)=\begin{bmatrix}0 & 1 \\-1 & \alpha \end{bmatrix}\begin{bmatrix}0 & 1 \\-1 & \beta \end{bmatrix}\begin{bmatrix}a & b \\c & d \end{bmatrix}=\begin{bmatrix} -a+\beta c & -b+\beta d \\ -\alpha a-c+\alpha\beta c & -\alpha b-d+\alpha\beta d \end{bmatrix}
\]
and
\[
\pi(\gamma x)=\begin{bmatrix}0 & 1 \\-1 & \gamma \end{bmatrix}\begin{bmatrix} a & b \\ c & d\end{bmatrix} = \begin{bmatrix}c & d \\-a+\gamma c & -b+\gamma d\end{bmatrix}.
\]
Thus
\begin{align*}
\alpha\beta x \in \mathcal{A}_r &\Leftrightarrow (-b+\beta d)r = -\alpha b-d+\alpha\beta d\\
&\Leftrightarrow -br+\beta dr = -\alpha b-d+\alpha\beta d\\
&\Leftrightarrow dr^2 - \alpha dr = -br + \alpha b + dr^2 - (\alpha-\beta)dr + (1-\alpha\beta)d\\
&\Leftrightarrow dr = -b + \frac{(r^2 - (\alpha-\beta)r + 1 -\alpha\beta)}{r-\alpha}d,
\end{align*}
so $\alpha\beta x \sim_r \gamma x$ where $\gamma = \dfrac{(r^2 - (\alpha-\beta)r + 1 -\alpha\beta)}{r-\alpha}$. \\
(ii) Since ${\small\pi(r\beta x) = \begin{bmatrix}0 & 1 \\-1 & r \end{bmatrix}\begin{bmatrix}0 & 1 \\-1 & \beta \end{bmatrix}\begin{bmatrix} a & b \\ c & d\end{bmatrix} = \begin{bmatrix}-a+\beta c & -b+\beta d \\-c-ar+\beta rc & -d-br+\beta rd\end{bmatrix}}$, $r\beta x \in \mathcal{A}_r \Leftrightarrow (-b+\beta d)r = -d-br+\beta rd \Leftrightarrow d=0 \Leftrightarrow x \in \mathcal{A}_0$.
\end{proof}

\begin{remark}
This result leads to an algorithm to distinguish words in $F_k$. It extends Bacher's work on $\sim_0$ in \cite{B02} Proposition 2.4 (ii) to $\sim_r, r \in k$. Note that $\alpha \sim_r \varepsilon \Leftrightarrow \alpha \ne r$. 
Combined with Theorem \ref{main}, we completely classify all words into the partition $\mathcal{A}_r$ and $\mathcal{C}_r$ of $F_k$. 
\end{remark}

We illustrate Theorem \ref{main} and the above remark by the following numerical example.

\begin{ex}
Let $k = \mathbb{F}_3$. Consider $22102 \in F_k$.
\begin{itemize}	
	\item[$r=0$.] By Theorem \ref{main} (i), $22102 \sim_0 (2-2^{-1})102 = 0102.$
	By Theorem \ref{main} (ii), $0102 \sim_0 02 \sim_0 \varepsilon.$	Then we have $22102 \in \mathcal{C}_0$.
	
	\item[$r=1$.] 	By Theorem \ref{main} (i),		
	\begin{align*}
	22102 &\sim_1 \left[\frac{1^2-(2-2)1+1-2\cdot2}{1-2}\right]102 = 2102\\
	&\sim_1 \left[\frac{1^2-(2-1)1+1-2\cdot1}{1-2}\right]02 = 102.
	\end{align*}
	Since $2 \in \mathcal{C}_0$, $102 \in \mathcal{C}_1$ by Theorem \ref{main} (ii). Then we have $22102 \in \mathcal{C}_1$.
	
\item[$r=2$.]
	By Theorem \ref{main} (ii), we first consider 
	\[
	102 \sim_0 (0-1^{-1})2 = 22
	\sim_0 (2-2^{-1})\varepsilon = 0.
	\]
	Then $102 \in \mathcal{A}_0$, so we have $22102 \in \mathcal{A}_2$.	
\end{itemize}
\end{ex}

\bigskip

\section{Words over $\mathbb{Z}/N\mathbb{Z}$}

In this section, we study arithmetic properties of the partition $\bar{\mathcal{A}}$ and $\bar{\mathcal{C}}$ of the set $F_N$ defined parallel to Bacher's. Let
\[
S' =  \left\{ \begin{bmatrix} 0 & 1 \\ -1 & \alpha \end{bmatrix} : \alpha \in \bbZ/N\bbZ \right\}.
\]
We begin by giving the proof of the following lemma.

\begin{lemma} \label{lemsemi}
The set $S'$ generates $\SL_2(\bbZ/N\bbZ)$ as a semigroup. 
\end{lemma}
\begin{proof}
Recall Theorem 2 in Chapter VII of Serre's book \cite{S73} that the set of matrices $\left\{ \begin{bmatrix}-1 & -1 \\0 & -1 \end{bmatrix}, \begin{bmatrix}0 & 1 \\-1 & 0 \end{bmatrix} \right\}$ generates $\SL_2(\mathbb{Z})$ as a group. Since the map $\SL_2(\mathbb{Z}) \to \SL_2(\mathbb{Z}/N\mathbb{Z})$ obtained by reducing the matrix entries modulo~$N$ is a surjective group homomorphism. Then $\left\{\begin{bmatrix}-1 & -1 \\0 & -1 \end{bmatrix}, \begin{bmatrix}0 & 1 \\-1 & 0 \end{bmatrix} \right\} \mod N$ also generates $\SL_2(\mathbb{Z}/N\mathbb{Z})$ as a group.

Consider $\langle S' \rangle$, a semigroup generated by $S'$. Since $\SL_2(\bbZ/N\bbZ)$ is finite, $\langle S' \rangle$ is a finite closed subset of $\SL_2(\bbZ/N\bbZ)$, so it is a subgroup. Note that $\langle S' \rangle$ contains both generators $\begin{bmatrix}0 & 1 \\-1 & 0 \end{bmatrix}$ and $\begin{bmatrix}-1 & -1 \\0 & -1 \end{bmatrix}=\begin{bmatrix}0 & 1 \\-1 & 1 \end{bmatrix}\begin{bmatrix}0 & 1 \\-1 & 1 \end{bmatrix}\begin{bmatrix}0 & 1 \\-1 & 0 \end{bmatrix}$ of $\SL_2(\mathbb{Z}/N\mathbb{Z})$. Hence $\langle S' \rangle = \SL_2(\mathbb{Z}/N\mathbb{Z})$.
\end{proof}

This lemma shows that every element of $\SL_2(\bbZ/N\bbZ)$ can be written in at least one way as a finite word with letters in $S'$.

Next, we establish a way to determine if words are in $\bar{\mathcal{A}}$. For $w \in F_N$ with $\pi(w) = \begin{bmatrix} a & b \\c & d \end{bmatrix} \in \SL_2(\bbZ/N\bbZ)$, we note that
\begin{align*}
w \in \bar{\mathcal{A}} 
&\Leftrightarrow \begin{bmatrix} 1 \\ 0 \end{bmatrix}  = \begin{bmatrix} a & b \\c & d \end{bmatrix} \begin{bmatrix} 0 \\ 1 \end{bmatrix} = \begin{bmatrix} b \\ d \end{bmatrix} \Leftrightarrow d=0 \\ 
& \Leftrightarrow \pi(w) = \begin{bmatrix}a & b \\ -b^{-1} & 0 \end{bmatrix} \; \text{with $a \in \bbZ/N\bbZ, b \in (\bbZ/N\bbZ)^{\times}$}.
\end{align*}
Hence we have shown

\begin{thm}
Let $N$ be a positive integer. Then
\[
\bar{\mathcal{A}} = \left\{w \in F_N : \pi(w)=\begin{bmatrix}a & b \\-b^{-1} & 0 \end{bmatrix} \text{for some} \ a \in \bbZ/N\bbZ, b \in (\bbZ/N\bbZ)^{\times}\right\}.
\]
\end{thm}

For $l \ge 0$, we write $F_N^l$ for the set of words over $\bbZ/N\bbZ$ of length $l$, $\bar{\mathcal{A}}^l = F_N^l \cap \bar{\mathcal{A}}$ and $\bar{\mathcal{C}}^l = F_N^l \cap \bar{\mathcal{C}}$. We first study the insertion and deletion in $\bar{\mathcal{A}}$.

\begin{thm}
If $w \in \bar{\mathcal{A}}^l$, then $\alpha w \in \bar{\mathcal{C}}^{l+1}$ and $w\alpha \in \bar{\mathcal{C}}^{l+1}$ for every $\alpha \in \mathbb{Z}/N\mathbb{Z}$.
\end{thm} 
\begin{proof}
Assume that $w \in \bar{\mathcal{A}}^l$ and let $\alpha \in \bbZ/N\bbZ$. Then $\pi(w)=\begin{bmatrix} a & b \\-b^{-1} & 0 \end{bmatrix}$ where $a \in \mathbb{Z}/N\mathbb{Z}$ and $b \in (\mathbb{Z}/N\mathbb{Z})^\times$. Thus 
\[
\pi(\alpha w)=\pi(\alpha)\pi(w)=\begin{bmatrix} 0 & 1 \\-1 & \alpha \end{bmatrix}\begin{bmatrix} a & b \\-b^{-1} & 0 \end{bmatrix}=\begin{bmatrix} -b^{-1} & 0 \\-a-\alpha b^{-1} & -b \end{bmatrix}
\]
and 
\[
\pi(w\alpha)=\pi(w)\pi(\alpha)=\begin{bmatrix} a & b \\-b^{-1} & 0 \end{bmatrix}\begin{bmatrix} 0 & 1 \\-1 & \alpha \end{bmatrix}=\begin{bmatrix} -b & a+\alpha b \\0 & -b^{-1} \end{bmatrix}.
\] 
Since $b \ne 0$, $\alpha w \in \bar{\mathcal{C}}^{l+1}$ and $w\alpha \in \bar{\mathcal{C}}^{l+1}$.
\end{proof}
 
\begin{thm} \label{zeroeffect}
Let $w \in \bar{\mathcal{C}}^l$ with $\pi(w)=\begin{bmatrix} a & b \\ c & d \end{bmatrix} \in \SL_2(\mathbb{Z}/N\mathbb{Z})$. \\ 
{\rm (i)} If $\gcd(d,N) = 1$, i.e., $d \in (\mathbb{Z}/N\mathbb{Z})^\times$, then there exist unique $\alpha ,\beta \in \mathbb{Z}/N\mathbb{Z}$ such that $\alpha w \in \bar{\mathcal{A}}^{l+1}$ and $w\beta \in \bar{\mathcal{A}}^{l+1}$. \\
{\rm (ii)} If $\gcd(d,N)>1$, then $\alpha w \in \bar{\mathcal{C}}^{l+1}$ and $w\beta \in \bar{\mathcal{C}}^{l+1}$ for all $\alpha, \beta \in \mathbb{Z}/N\mathbb{Z}$.
\end{thm} 
\begin{proof}
We first note that for $\alpha \in \mathbb{Z}/N\mathbb{Z}$, $$\pi(\alpha w)=\pi(\alpha)\pi(w)=\begin{bmatrix} 0 & 1 \\-1 & \alpha \end{bmatrix}\begin{bmatrix} a & b \\c & d \end{bmatrix}=\begin{bmatrix} c & d \\-a+\alpha c & -b+\alpha d \end{bmatrix}.$$ Then $\alpha w \in \bar{\mathcal{A}}^{l+1} \Leftrightarrow -b+\alpha d \equiv 0 \mod N$. 
This congruence equation has a solution $\Leftrightarrow \gcd(d,N)|b$. We claim that $\gcd(d,N)|b$ is equivalent to $\gcd(d,N)=1$ and the theorem can easily be deduced. It is obvious that $\gcd(d,N)=1$ implies $\gcd(d,N)|b$. If $\gcd(d,N)|b$, then $\gcd(d,N)$ is a common divisor of $d$ and $b$. Since $ad - bc = 1$, $\gcd(d,N)\leq \gcd(d,b)=1$, so $\gcd(d,N)=1$.  Hence we have the claim.
\end{proof}
 
\begin{thm}
If $\alpha_1 \dots \alpha_l \in \bar{\mathcal{A}}^l$, then $\alpha_2 \dots \alpha_l$ and $\alpha_1 \dots \alpha_{l-1} \in \bar{\mathcal{C}}^{l-1}$.
\end{thm} 
\begin{proof}
Assume that $\alpha_1\ldots\alpha_l \in \bar{\mathcal{A}}^l$. Then $\pi(\alpha_1\ldots\alpha_l)=\begin{bmatrix} a & b \\ -b^{-1} & 0 \end{bmatrix}$ for some $a \in \mathbb{Z}/N\mathbb{Z}$ and $b \in (\mathbb{Z}/N\mathbb{Z})^\times$. Thus
\[
\pi(\alpha_2\ldots\alpha_l)=\pi(\alpha_1)^{-1} \pi(\alpha_1\ldots\alpha_l) 
=\begin{bmatrix} \alpha_1 & -1 \\ 1 & 0 \end{bmatrix}\begin{bmatrix} a & b \\ -b^{-1} & 0 \end{bmatrix}
=\begin{bmatrix} \alpha_1 a+b^{-1} & \alpha_1 b \\ a & b \end{bmatrix}
\]
and 
\[
\pi(\alpha_1\ldots\alpha_{l-1})=\pi(\alpha_1\ldots\alpha_l)\pi(\alpha_l)^{-1}
=\begin{bmatrix} a & b \\ -b^{-1} & 0 \end{bmatrix}\begin{bmatrix} \alpha_l & -1 \\ 1 & 0 \end{bmatrix} =\begin{bmatrix} \alpha_l a+b &  -a \\ -\alpha_l b^{-1} & b^{-1} \end{bmatrix}.
\]
Since $b \ne 0$, $\alpha_2\ldots\alpha_l$ and $\alpha_1\ldots\alpha_{l-1}\in \bar{\mathcal{C}}^{l-1}$.
\end{proof}

\begin{remark}
We used to be able to derive the cardinalities of $\mathcal{A}_r$ and  $\mathcal{C}_r$ by knowing the properties given in the above three theorem. However, in the finite ring $\bbZ/N\bbZ$ case is not the same as the finite field $k$ case due to this ring contains zero divisors. This makes the words in $F_N$ behave differently as we have seen in Theorem \ref{zeroeffect}. 
\end{remark}

Another property of words in $\bar{\mathcal{A}}$ is given in the following theorem. This result will be used in the next section.

\begin{thm} \label{reverse}
$\alpha_1\alpha_2 \dots \alpha_l \in \bar{\mathcal{A}}^l$ if and only if $\alpha_l\alpha_{l-1} \dots \alpha_1 \in \bar{\mathcal{A}}^l$.
\end{thm} 
\begin{proof}
Consider $\sigma = \begin{bmatrix} 0 & 1 \\ 1 & 0 \end{bmatrix}$ and $\alpha_1\ldots\alpha_l \in F_N$ with $\pi(\alpha_1\ldots\alpha_l)=\begin{bmatrix} a & b \\ c & d \end{bmatrix} \in \SL_2(\mathbb{Z}/N\mathbb{Z})$. 
Since $\sigma\begin{bmatrix} w & x \\ y & z \end{bmatrix}\sigma=\begin{bmatrix} z & y \\ x & w \end{bmatrix}$ for all $w,x,y,z\in\mathbb{Z}/N\mathbb{Z}$ and $\sigma = \sigma^{-1}$, we have
\begin{align*}
\begin{bmatrix} d & c \\ b & a \end{bmatrix}
&=\sigma\pi(\alpha_1 \alpha_2 \dots \alpha_l)\sigma =(\sigma\pi(\alpha_1)\sigma)(\sigma\pi(\alpha_2)\sigma)\dots(\sigma\pi(\alpha_l)\sigma) \\
&=\begin{bmatrix} \alpha_1 & -1 \\ 1 & 0 \end{bmatrix} \dots\begin{bmatrix} \alpha_l & -1 \\ 1 & 0 \end{bmatrix} 
=\left(\begin{bmatrix} 0 &  1\\ -1 & \alpha_l \end{bmatrix}\ldots\begin{bmatrix} 0 &  1\\ -1 & \alpha_1 \end{bmatrix}\right)^{-1} \\
&=\pi(\alpha_l\ldots\alpha_1)^{-1},
\end{align*}
so $\pi(\alpha_l\dots\alpha_1)=\begin{bmatrix} a &  -c\\ -b & d \end{bmatrix}$.
Thus $\alpha_1\alpha_2\dots\alpha_l \in \bar{\mathcal{A}}^l \Leftrightarrow d=0 \Leftrightarrow \alpha_l\alpha_{l-1} \ldots \alpha_1\in ~\bar{\mathcal{A}}^l$.
\end{proof}

The partition $\bar{\mathcal{A}}$ and $\bar{\mathcal{C}}$ of $F_k$ also induces the equivalence relation $\sim$ on $F_k$. We record some relationships between two words in the next theorem.

\begin{thm}
Let $x \in F_N$ and $\beta \in \mathbb{Z}/N\mathbb{Z}$. We have \\
{\rm(i)} If $\alpha \in (\mathbb{Z}/N\mathbb{Z})^{\times}$, then $\alpha\beta x \sim (\beta -\alpha^{-1}) x$. \\
{\rm(ii)} $0\beta x \sim x$.
\end{thm} 
\begin{proof} 
Let $\pi(x)=\begin{bmatrix}a & b \\c & d \end{bmatrix} \in \SL_2(\mathbb{Z}/N\mathbb{Z})$. \\
(i) Assume that $\alpha \in (\mathbb{Z}/N\mathbb{Z})^{\times}$. Then
\[
\pi(\alpha\beta x)=\begin{bmatrix}0 & 1 \\-1 & \alpha \end{bmatrix}\begin{bmatrix}0 & 1 \\-1 & \beta \end{bmatrix}\begin{bmatrix}a & b \\c & d \end{bmatrix}=\begin{bmatrix} -a+\beta c & -b+\beta d \\ -\alpha a-c+\alpha\beta c & -\alpha b-d+\alpha\beta d \end{bmatrix}
\]
and
\[
\pi((\beta -\alpha^{-1}) x)=\begin{bmatrix}0 & 1 \\-1 & \beta -\alpha^{-1} \end{bmatrix}\begin{bmatrix} a & b \\ c & d\end{bmatrix} = \begin{bmatrix}c & d \\-a+(\beta -\alpha^{-1}) c & -b+(\beta -\alpha^{-1}) d\end{bmatrix}.
\]
Thus
\begin{align*}
\alpha\beta x \in \bar{\mathcal{A}} &\Leftrightarrow -\alpha b-d+\alpha\beta d = 0\\
&\Leftrightarrow -b+(\beta -\alpha^{-1}) d = 0\\
&\Leftrightarrow (\beta -\alpha^{-1}) x \in \bar{\mathcal{A}},
\end{align*}
so $\alpha\beta x \sim (\beta -\alpha^{-1}) x$. \\
(ii) Since $\pi(0 \beta x) = \begin{bmatrix}0 & 1 \\-1 & 0 \end{bmatrix}\begin{bmatrix}0 & 1 \\-1 & \beta \end{bmatrix}\begin{bmatrix}a & b \\c & d \end{bmatrix} = \begin{bmatrix} -a+\beta c & -b+\beta d \\ -c & -d \end{bmatrix}$, $0 \beta x \in~\bar{\mathcal{A}}$  $\Leftrightarrow  d = 0 \Leftrightarrow x \in \bar{\mathcal{A}}$, so $0\beta x \sim x$.
\end{proof}

\begin{remark}
The above theorem yields partial answers (again due to zero divisors in $\bbZ/N\bbZ$) for determination of words into classes $\bar{\mathcal{A}}$ and $\bar{\mathcal{C}}$. However, a good mathematical software such as Maple${\sf ^{TM}}$ can easily compute the product of $2 \times 2$ matrices modulo positive integer~$N$. This allows us to directly distinguish words in $F_N$. 
\end{remark}

\bigskip

\section{More on $\bar{\mathcal{A}}$}

We concentrate more on $\bar{\mathcal{A}}$ and record its further parallel properties to Bacher's in this last section. This work includes unique factorization, predecessors, successors and periodic words. 

In order to prove the fact about unique factorization on $\bar{\mathcal{A}}$, we start with the following lemma.

\begin{lemma} \label{one}
{\rm (i)} If $w,w' \in \bar{\mathcal{A}}$ then $ww' \in \bar{\mathcal{C}}$ and $w\alpha w' \in \bar{\mathcal{A}}$ for any
$\alpha \in \mathbb{Z}/N \mathbb{Z}$. \\
{\rm (ii)} If exactly one of $w,w'$ is an element of $\bar{\mathcal{A}}$ then $w\alpha w' \in \bar{\mathcal{C}}$ for any
$\alpha \in \mathbb{Z}/N \mathbb{Z}$. 
\end{lemma}
\begin{proof} 
To prove (i), let $\pi(w)=\begin{bmatrix}a & b \\-b^{-1} & 0 \end{bmatrix}$ and $\pi(w')=\begin{bmatrix}a' & b' \\-b'^{-1} & 0 \end{bmatrix}$ for some $a, a' \in \mathbb{Z}/N \mathbb{Z}$, $b,b' \in (\mathbb{Z}/N\mathbb{Z})^\times$, and let $\alpha \in \mathbb{Z}/N \mathbb{Z}$. Then
\[
\pi(ww') 
=\begin{bmatrix}a & b \\-b^{-1} & 0 \end{bmatrix}\begin{bmatrix}a' & b' \\-b'^{-1} & 0 \end{bmatrix}
=\begin{bmatrix} aa'-bb'^{-1} & ab' \\ -a'b^{-1} & -b^{-1}b' \end{bmatrix},
\]
and
\[
\pi(w\alpha w') 
= \begin{bmatrix}a & b \\-b^{-1} & 0 \end{bmatrix}\begin{bmatrix}0 & 1 \\-1 & \alpha \end{bmatrix} 
\begin{bmatrix}a' & b' \\-b'^{-1} & 0 \end{bmatrix}
=\begin{bmatrix}-ba'-ab'^{-1}-\alpha bb'^{-1} & -bb' \\(bb')^{-1} & 0 \end{bmatrix}.	
\]
Thus $w\alpha w' \in \bar{\mathcal{A}}$ for any $\alpha \in \mathbb{Z}/N \mathbb{Z}$. Since $b,b' \in (\mathbb{Z}/N\mathbb{Z})^\times$, $ww' \in \bar{\mathcal{C}}$.

To prove (ii), suppose that $w \in \bar{\mathcal{A}}$ and $w' \in \bar{\mathcal{C}}$ and let $\alpha \in \mathbb{Z}/N \mathbb{Z}$. Then $\pi(w) = \begin{bmatrix} a & b \\-b^{-1} & 0 \end{bmatrix}$ for some $a \in \mathbb{Z}/N\mathbb{Z}$ and $b \in (\mathbb{Z}/N\mathbb{Z})^\times$, and $\pi(w')=\begin{bmatrix} a' & b' \\c' & d' \end{bmatrix}$ in $\SL_2(\mathbb{Z}/N \mathbb{Z})$ with $d' \ne 0$. Thus
\[
{\small \pi(w\alpha w')=\begin{bmatrix}a & b \\-b^{-1} & 0 \end{bmatrix}\begin{bmatrix}0 & 1 \\-1 & \alpha \end{bmatrix}\begin{bmatrix}a' & b' \\c' & d' \end{bmatrix}=\begin{bmatrix}-ba'+ac'+\alpha bc' & -bb'+ad'+\alpha bd' \\-b^{-1}c' & -b^{-1}d' \end{bmatrix}.}
\]
Since $b \in (\mathbb{Z}/N\mathbb{Z})^\times$ and $d'\neq 0$, $w\alpha w' \in \bar{\mathcal{C}}$. For another case, let $w=\beta_1 \ldots \beta_m~\in~\bar{\mathcal{C}}$ and $w'=\beta'_1 \ldots \beta'_n \in \bar{\mathcal{A}}$ for some positive integers $m$ and $n$. By Theorem \ref{reverse}, we have $\beta'_n \ldots \beta'_1 \in \bar{\mathcal{A}}$ and $\beta_m \ldots \beta_1 \in \bar{\mathcal{C}}$. The previous proof shows that $$\beta'_n \ldots \beta'_1 \alpha \beta_m \ldots \beta_1 \in \bar{\mathcal{C}}.$$ Thus we get $w\alpha w' = \beta_1 \ldots \beta_m \alpha \beta'_1 \ldots \beta'_n \in \bar{\mathcal{C}}$ by Theorem \ref{reverse}. 
\end{proof}

Let $\mathcal{P}^l = \left\{ \alpha_1 \alpha_2 \ldots \alpha_l \in \bar{\mathcal{A}}^l : \alpha_1 \ldots \alpha_h \in \bar{\mathcal{C}}^h \, \text{for} \ h=1, \ldots , l-1 \right\}$ and $\mathcal{P}=\bigcup \mathcal{P}^l$.

\begin{thm} 
{\rm [Unique Factorization in $\bar{\mathcal{A}}$]} 
Let $w \in F_N$. Then  $w \in \bar{\mathcal{A}}$ if and only if $w$ can be written as 
\[
w = p_1\delta_1p_2\delta_2\ldots p_n\delta_np_{n+1} 
\]
for some $n \geq 0$ with $p_1, \ldots, p_{n+1} \in \mathcal{P}$ and $\delta_1, \ldots, \delta_n \in \mathbb{Z}/N\mathbb{Z}$. Moreover, such a factorization of $w \in \bar{\mathcal{A}}$ is unique.
\end{thm} 
\begin{proof}
Suppose that $w$ can be written as in this form. By Lemma \ref{one}, it is easy to see that $w \in \bar{\mathcal{A}}$. Conversely, assume that $w = \alpha_1 \alpha_2 \ldots \alpha_l \in \bar{\mathcal{A}}^l$. Then there is the smallest positive integer $s$ such that $\alpha_1 \ldots \alpha_s \in \bar{\mathcal{A}}$. Setting $p_1=\alpha_1\ldots\alpha_s \in \mathcal{P}$ and $\delta_1=\alpha_{s+1}$. Thus $\alpha_{s+2} \alpha_{s+3} \ldots \alpha_l$ must be in $\bar{\mathcal{A}}^{l-(s+1)}$ by Lemma \ref{one}. Repeating this process we get the sets $\{\delta_1, \ldots, \delta_n\} \subset \mathbb{Z}/N\mathbb{Z}$ and $\{p_1,\ldots,p_{n+1}\} \subset \mathcal{P}$ so that $w = p_1 \delta_1 p_2 \delta_2 \ldots p_n \delta_np_{n+1}$ for some $n \geq 0$. The smallest length of $p_i$ for each $i$ implies the uniqueness of this factorization.
\end{proof}

Given two words $w,w' \in F_N$ of the form
\[
w = \alpha_0\alpha_1\ldots\alpha_{l-1} \quad \text{and}  \quad w' = \alpha_1\alpha_2\ldots\alpha_l,
\] 
we call $w'$ an {\it immediate successor} of $w$ and $w$ an {\it immediate predecessor} of $w'$.

\begin{thm} \label{imm} 
Each element $w \in \bar{\mathcal{A}}^l$ has a unique immediate successor and a unique immediate predecessor in $\bar{\mathcal{A}}^l$.
\end{thm} 
\begin{proof}
Assume that $\alpha_0 \alpha_1 \ldots \alpha_{l-1} \in \bar{\mathcal{A}}^l$. Then $\pi(\alpha_0 \alpha_1 \ldots \alpha_{l-1}) = \begin{bmatrix}
a & b \\-b^{-1} & 0 \end{bmatrix}$ for some $a \in \mathbb{Z}/N \mathbb{Z}$ and $b \in (\mathbb{Z}/N\mathbb{Z})^\times$. Thus
\begin{align*}
\pi(\alpha_1 \alpha_2 \ldots \alpha_{l-1}) 
&= \pi(\alpha_0)^{-1}\pi(\alpha_0 \alpha_1 \ldots \alpha_{l-1})\\
&= \begin{bmatrix} \alpha_0 & -1 \\1 & 0 \end{bmatrix}\begin{bmatrix} a & b \\-b^{-1} & 0 \end{bmatrix}
= \begin{bmatrix} \alpha_0 a + b^{-1} & \alpha_0 b \\a & b \end{bmatrix},
\end{align*}
so
\begin{align*}
\pi(\alpha_1 \alpha_2 \ldots \alpha_l) &=\pi(\alpha_1 \alpha_2 \ldots \alpha_{l-1})\pi(\alpha_l)\\ 
&= \begin{bmatrix} \alpha_0 a+b^{-1} & \alpha_0 b \\a & b \end{bmatrix} \begin{bmatrix}0 & 1 \\-1 & \alpha_l \end{bmatrix} 
=\begin{bmatrix} -\alpha_0 b & \alpha_0 a + b^{-1} + \alpha_0\alpha_lb \\-b & a + \alpha_lb \end{bmatrix}.
\end{align*}
Since $b \in (\mathbb{Z}/N \mathbb{Z})^\times$, $\alpha_1 \alpha_2 \ldots \alpha_l \in \bar{\mathcal{A}}^l \Leftrightarrow \alpha_l = -ab^{-1}$. Hence $w$ has a unique immediate successor in $\bar{\mathcal{A}}^l$. Similarly, we can show  that $w$ also has a unique immediate predecessor in $\bar{\mathcal{A}}^l$.
\end{proof}

For $w = \alpha_1 \alpha_2 \ldots \alpha_l\in \bar{\mathcal{A}}^l$, by Theorem \ref{imm} there exists an infinite word
\[W=\ldots \alpha_{-1} \alpha_0 \alpha_1 \alpha_2 \alpha_3 \ldots\]
such that $\alpha_{i+1} \ldots \alpha_{i+l}$ is the immediate successor in $\bar{\mathcal{A}}^l$ of $\alpha_{i} \ldots \alpha_{i+l-1}$ for all integer~$i$. That is, all subwords formed by $l$ consecutive letters of $W$ are elements in $\bar{\mathcal{A}}^l$. Since $\bar{\mathcal{A}}^l$ is finite, the infinite word $W$ associated to $w$ is periodic. Hence for every $w \in \bar{\mathcal{A}}^l$, there exists the smallest positive integer $s$ such that the infinite word $W$ associated to $w$ is $s$-periodic.    

\begin{ex}
Some infinite periodic words over $\mathbb{Z}/6\mathbb{Z}$.
\begin{enumerate}
	\item The infinite periodic word corresponding to both $121$ and $212$ is a $2$-periodic word $\ldots 1212 \ldots$. 
	\item The infinite periodic word corresponding to $234, 343$ and $432$ is a $4$-periodic word 
				$\ldots 23432343 \ldots$.
\end{enumerate}
\end{ex}

\begin{thm}
Let $W=\ldots \alpha_{s-1} \alpha_0 \alpha_1 \ldots \alpha_{s-1} \alpha_0 \alpha_1 \ldots$ be an infinite $s$-periodic word with letters in $\mathbb{Z} / N \mathbb{Z}$. Then there exists a smallest positive integer $t$ such that all
subwords of length $ts-1$ in $W$ belong to $\bar{\mathcal{A}}$. Moreover, all subwords of length $lts-1 \ (l \geq 1)$ of $W$ belong to $\bar{\mathcal{A}}$.
\end{thm}
\begin{proof}
We observe that the elements
\[\pi(\alpha_0 \alpha_1 \ldots \alpha_{s-1}), \pi(\alpha_1 \ldots \alpha_{s-1} \alpha_0), \ldots , \pi(\alpha_{s-1} \alpha_0 \ldots \alpha_{s-2}) \in \SL_2(\mathbb{Z}/N \mathbb{Z})\] 
are all conjugate. Then they have a common order $t'$, we claim that $t'$ has the desired property. Let $w$ be a subword of length $t's$ in $W$. Thus $w = \underbrace{w'w' \ldots w'}_{t' \ \text{copies}}$ where $w'$ is a subword of length $s$ in $W$, so $$\pi(w) = \pi(\underbrace{w'w' \ldots w'}_{t' \ \text{copies}}) = (\pi(w'))^{t'} = \begin{bmatrix} 1 & 0 \\0 & 1 \end{bmatrix}.$$ Assume that $w = \beta_1 \ldots \beta_{t's}$. The subword of length $t's-1$ associated with $w$ is in~$\bar{\mathcal{A}}$ as a result of $\pi(\beta_1 \ldots \beta_{t's-1}) = \pi(w) \pi(\beta_{t's})^{-1} = \begin{bmatrix} \beta_{t's} & -1 \\1 & 0 \end{bmatrix}$. Hence $t \le t'$ exists by the well-ordering principle.
\end{proof}

\begin{remark}
In the above proof, sometimes $t < t'$. For example, consider the infinite $1$-periodic word, $\ldots 000 \ldots$. The order of $\pi(0) = 4$ but we can choose $t=2$ since $0 \in \bar{\mathcal{A}}$. Moreover, since $t'$ divides $|\SL_2(\mathbb{Z} / N \mathbb{Z})|$, we know that $t \le |\SL_2(\mathbb{Z} / N \mathbb{Z})|$.
\end{remark}

\begin{ex}
In $\mathbb{Z}/6\mathbb{Z}$, consider the $4$-periodic word $W = \ldots23432343\ldots$.
\begin{enumerate}
\item We have $t=1$ so that all subwords of length $4(1) - 1=3$ in $W$ belong to $\bar{\mathcal{A}}$.\\
$(234, 343, 432, 323)$
\item For $l=2$, all subwords of length $4(2)-1=7$ in $W$ belong to $\bar{\mathcal{A}}$.\\
$(2343234, 3432343, 4323432, 3234323)$
\item For $l=3$, all subwords of length $4(3)-1=11$ in $W$ belong to $\bar{\mathcal{A}}$.\\
$(23432343234, 34323432343, 43234323432, 32343234323)$
\end{enumerate}
\end{ex}


\begin{thebibliography}{99}

\bibitem[B00]{B00} R. Bacher,
\emph{An equivalence relation on $\{0, 1\}^*$}, 
Europ. J. Combinatorics,
\textbf{21} (2000), 853-864.

\bibitem[B02]{B02} R. Bacher,
\emph{$\SL_2(k)$ and a subset of words over $k$}, 
Europ. J. Combinatorics,
\textbf{23} (2002), 141-147.

\bibitem[S73]{S73} J.-P. Serre,
\emph{A Course in Arithmetic}, 
Springer-Verlag, New York,
1973.

\end{thebibliography}
\end{document}